\documentclass[12pt,reqno]{amsart}
\usepackage{amscd}
\usepackage{amsmath}
\usepackage{latexsym}
\usepackage{amsfonts}
\usepackage{amssymb}
\usepackage{amsthm}
\usepackage{graphicx}
\usepackage{verbatim}
\usepackage{mathrsfs}
\usepackage{enumerate}
\usepackage{color}

\newtheorem{theorem}{Theorem}[section]

\theoremstyle{remark}

\newcommand{\beq}{\begin{equation}}
\newcommand{\eeq}{\end{equation}}
\newcommand{\ben}{\begin{eqnarray}}
\newcommand{\een}{\end{eqnarray}}
\newcommand{\beno}{\begin{eqnarray*}}
\newcommand{\eeno}{\end{eqnarray*}}
\def\na{\nabla}
\newcommand{\fr}{\frac}
\newcommand{\pa}{\partial}

\newcommand{\la}{\label}
\newcommand{\ba}{\begin{array}{l}}
\newcommand{\ea}{\end{array}}

\newcommand{\RR}{\mathbb{R}}

\allowdisplaybreaks

\begin{document}

\title[The Oldroyd-B model]{High Reynolds number and high Weissenberg number Oldroyd-B model with dissipation}

\author[Peter Constantin, Jiahong Wu, Jiefeng Zhao and Yi Zhu]
{Peter Constantin$^{1}$,  Jiahong Wu$^{2}$, Jiefeng Zhao$^{3}$ and Yi Zhu$^{4}$}

\address{$^1$ 
	Department of Mathematics, 
	Princeton University,  
	Fine Hall, Washington Road,
	Princeton, NJ 08544-1000,  United States}

\email{const@math.princeton.edu}

\address{$^2$ Department of Mathematics, Oklahoma State University, Stillwater, OK 74078, United States}

\email{jiahong.wu@okstate.edu}

\address{$^3$ School of Mathematics and Information Science, 
	Henan Polytechnic University, Jiaozuo 454003, P. R. China }

\email{zhaojiefeng003@hpu.edu.cn}

\address{$^4$ School of Science, East China University of Science and Technology, Shanghai 200237, P. R. China}

\email{zhuyim@ecust.edu.cn}

\vskip .2in
\begin{abstract}
We give a small data global well-posedness result
for an incompressible Oldroyd-B model with wave number dissipation in the equation of stress tensor. The result is uniform in solvent Reynolds numbers, and requires only fractional wave-number dependent dissipation $(-\Delta)^{\beta}$, $\beta \ge \fr{1}{2}$ in the added stress.
\end{abstract}
\maketitle

\vskip .1in
\section{Introduction}
\setcounter{equation}{0}
A class of models of complex fluids is based on  an equation for a solvent coupled with a kinetic description of particles suspended in it. In the case of dilute suspensions weakly confined by a Hookean spring potential, a rigorously established exact closure for the moments in the kinetic equation of this Navier-Stokes-Fokker-Planck system yields the Oldroyd-B system (\cite{La}). 
After non-dimensionalization, the coupled Oldroyd-B system is
\begin{equation}\la{nsigma}
\begin{cases}
\partial_t u + u\cdot\nabla u + \nabla p - \frac{1}{Re}\Delta u  =  K\nabla\cdot \sigma,\\
\partial_t\sigma + u\cdot\nabla \sigma = (\nabla u)\sigma + \sigma (\nabla u)^* - \frac{1}{We}(\sigma - \mathbb I),\\
\na\cdot u = 0,
\end{cases}
\end{equation}
where $\sigma$ is the conformation tensor, $\sigma = \mathbb E (m\otimes m)$ with $m$ the end-to-end vector in $\RR^d$ and $\mathbb E$  the average with respect to the local distribution, $u$ is the solvent velocity, $p$ is the pressure, $Re$ is the Reynolds number of the solvent, $We$ is  the Weissenberg number, $K = \fr{1}{\gamma Re We}$ and $\gamma$  is the ratio of solvent viscosity to polymeric viscosity. In the limit of zero Reynolds number, the system (\ref{nsigma}) reduces further and it becomes a nonlinear evolution for $\sigma$
\beq
\pa_t \sigma + u \cdot\na \sigma = (\na u)\sigma + \sigma(\na u)^* -\fr{1}{We}(\sigma - \mathbb I)
\la{sigmastokes}
\eeq
where $u$ is obtained from $\sigma$ by solving the Stokes system
\beq
-\Delta u +\na p = \fr{1}{\gamma We}\na\cdot \sigma, \quad\quad \na\cdot u = 0.
\la{stok}
\eeq
The system (\ref{sigmastokes}) with (\ref{stok}) is an example of an equation which might develop finite time singularities for large data, even in $\RR^2$. The forcing in the right hand side of (\ref{stok}) or in the right hand side of the momentum equation of (\ref{nsigma}) depends only on the added stress
\beq
\tau = \sigma - \mathbb I,
\la{add}
\eeq
because any multiple of the identitity matrix added to $\sigma$ is balanced by a pressure, even if the factor is a function of space and time. For small added stress it is known (\cite{Constantin1}) that the system (\ref{sigmastokes}), (\ref{stok}) has global solutions. The problem of global existence of smooth solutions for large data is open and challenging. The large Weissenberg number problem is challenging both numerically and analytically. If we replace the damping term by a wave-number dependent dissipative term we obtain an equation for the conformation stress
\beq
\pa_t \sigma + u \cdot\na \sigma = (\na u)\sigma + \sigma(\na u)^* -\eta P(D)(\sigma- \mathbb I)
\la{sigds}
\eeq
with $P(D)$ a dissipative differential operator and $\eta$ a positive number.
If a small diffusive term ($P(D) = -\Delta$ in (\ref{sigds})) is added to the equation for $\sigma$ coupled with (\ref{stok}) then global existence of smooth solutions with arbitrary data has been established (\cite{Constantin}) in $d=2$. For the small data problem one can discuss a less stringent wave-number dependence, and allow the solvent Reynolds number to be arbitrarily large.

In this paper we consider an Oldroyd-B model
\begin{equation}\label{OBA}
\begin{cases}
\partial_t u+u\cdot\nabla u+\nabla p=\nabla\cdot\tau,\ \ \ \ x\in\mathbb{R}^d,\ \ t>0,\\
\partial_t \tau+u\cdot\nabla\tau+\eta(-\Delta)^\beta\tau+Q(\tau,\nabla u)=D(u),  \\
\nabla\cdot u=0,\\
u(0,x)=u_0(x),\ \ \tau(0,x)=\tau_0(x),
\end{cases}
\end{equation}
where $0\leqslant\beta\leqslant1$ and $\eta>0$ are real parameters, $u=u(x,t)$ represents
the velocity field of the fluid, $p=p(x,t)$ the pressure and
$\tau=\tau(x,t)$ the non-Newtonian added stress tensor (see (\ref{add}))
(a $d$-by-$d$ symmetric matrix). Here $D(u)$ is the symmetric part of the velocity gradient, 
$$
D(u)=\frac{1}{2}\big(\nabla u+(\nabla
u)^\top\big)
$$
and the bilinear term $Q$ is taken to be
$$Q(\tau,\nabla u)=\tau W(u)-W(u)\tau-b\big(D(u)\tau+\tau D(u)\big)$$
with $b\in[-1,1]$ a constant and  $W(u)$ the skew-symmetric part of the $\nabla u$,
$$
W(u)=\frac{1}{2}\big(\nabla u-(\nabla u)^\top\big).
$$
The fractional Laplacian operator
$(-\Delta)^{\gamma}$ is defined through the Fourier transform, 
$$
\widehat{(-\Delta)^{\gamma} f}(\xi) = |\xi|^{2\gamma} \widehat{f}(\xi).
$$
For notational convenience, we also write $\Lambda=(-\Delta)^{\frac12}$ denoting the Zygmund operator. Background information on the Oldroyd-B model can be found in many references (see, e.g., \cite{R.B.Bird,J.G. Oldroyd}).

\vskip .1in
Our main result is the small data global well-posedness of \eqref{OBA} with any $\frac{1}{2}\leqslant \beta\leqslant1$. There is no damping mechanism in
the equation of $\tau$ in \eqref{OBA}: strictly speaking the Weissenberg number is infinite, but wave-number dependent dissipation is added. Whether or not \eqref{OBA} with $0\le \beta <\frac12$ possesses small data global well-posedness remains an open problem.

\begin{theorem} \label{main}
Consider \eqref{OBA} with $\frac{1}{2}\leqslant \beta\leqslant1$. Let $d=2,3$ and $s>1+\frac{d}{2}.$ Assume $(u_0,\tau_0)\in H^s(\mathbb{R}^d)$, $\nabla\cdot u_0=0$, and $\tau_0$ is symmetric. Then there
exists a small constant $\varepsilon>0$ such that, if
\begin{eqnarray*}
\|u_0\|_{H^s}+\|\tau_0\|_{H^s}\leqslant \varepsilon,
\end{eqnarray*}
then \eqref{OBA} has a unique global solution $(u,b)$ satisfying, for some constant $C>0$ and all $t>0$.
\begin{eqnarray*}
\|u\|_{H^s}+\|\tau\|_{H^s}\leqslant C\varepsilon.
\end{eqnarray*}
\end{theorem}

\vskip .1in
The small data global well-posedness for an Oldroyd-B model
without dissipation in the velocity equation has previously been examined by T. Elgindi and F. Rousset in the 2D case \cite{T.M.F} and by T. Elgindi and J. Liu for the 3D case \cite{T.M.J}. They focus on the following
Oldroyd-B model without velocity dissipation,
\begin{equation}\label{old}
\begin{cases}
u_t+u\cdot\nabla u+\nabla p=\nabla\cdot\tau,\ \ \ \ x\in\mathbb{R}^d,\ \ t>0,\\
\tau_t+u\cdot\nabla\tau+Q(\tau,\nabla u) -\eta \Delta \tau + a \tau=D(u),  \\
\nabla\cdot u=0,
\end{cases}
\end{equation}
where $a>0$ is a parameter. The small data global well-posedness result in \cite{T.M.F} is for (\ref{old}) with d=2 and $a>0$. The damping term plays a crucial role in the proof of their result and can not be removed. It was used to form a damping term in the equation of a combined quantity. \cite{T.M.J} examined (\ref{old}) with $d=3$ and $a>0$ and obtained the small data global well-posedness for any sufficiently small data $(u_0, \tau_0) \in H^3$. The damping term $a\tau$ in (\ref{old}) is also necessary for their result.

\vskip .1in
The velocity equation in (\ref{OBA}) is a forced Euler equation. As it is known, the $H^s$-norm of a solution of the Euler equation may grow in time, even perhaps at a double exponential rate (see, e.g., \cite{Den, Kis, Zla}). The Oldroyd-B system discussed has a dissipative structure, and a  main reason why  Theorem \ref{main} holds is a key observation on the linearized
system of (\ref{OBA}). Clearly, any solution $(u, \tau)$ of (\ref{OBA}) also solves
\begin{equation}\label{wave1}
\begin{cases}
\partial_t u+ \mathbb P (u\cdot\nabla u) =\mathbb P\nabla\cdot\tau,\ \ \ \ x\in\mathbb{R}^d,\ \ t>0,\\
\partial_t \mathbb P\nabla\cdot\tau + \mathbb P\nabla\cdot (u\cdot\nabla\tau) +\eta(-\Delta)^\beta\mathbb P\nabla\cdot\tau +\mathbb P\nabla\cdot Q(\tau,\nabla u)= \frac12 \Delta u,  \\
\nabla\cdot u=0,
\end{cases}
\end{equation}
where $\mathbb P$ denotes the Leray projection onto divergence-free
vector fields. The corresponding linearized system is given by
\begin{equation*}\label{wave2}
\begin{cases}
\partial_t u  =\mathbb P\nabla\cdot\tau,\\
\partial_t \mathbb P\nabla\cdot\tau +\eta(-\Delta)^\beta\mathbb P\nabla\cdot\tau = \frac12 \Delta u,  \\
\nabla\cdot u=0,
\end{cases}
\end{equation*}
which can be easily reduced to a system of decoupled wave type equations
\begin{equation}\label{wave3}
\begin{cases}
\partial_{tt} u + \eta (-\Delta)^\beta \partial_t u - \frac12 \Delta u=0,\\
\partial_{tt} (\mathbb P\nabla\cdot\tau) +\eta(-\Delta)^\beta \partial_t (\mathbb P\nabla\cdot\tau) - \frac12 \Delta (\mathbb P\nabla\cdot\tau) =0,  \\
\nabla\cdot u=0.
\end{cases}
\end{equation}
The structure in (\ref{wave3}) reveals that there are both dissipative and dispersive effects on $u$ in (\ref{OBA}). We remark that the Oldroyd-B model with only velocity dissipation share
a similar structure and has been shown by Yi Zhu to possess a unique global small solution \cite{zhuyi}. In order to prove the existence part of Theorem \ref{main}, we construct a suitable Lyapunov functional that
incorporate these effects.  We set the Lyapunov functional to be
$$
L(t) = \|u(t)\|_{H^s(\mathbb R^d)}^2 + \|\tau(t)\|_{H^s(\mathbb R^d)}^2 + 2 k (u(t), \nabla\cdot \tau(t))_{H^{s-\beta}(\mathbb R^d)},
$$
where $(f, g)_{H^\sigma(\mathbb R^d)}$ denotes the inner product in $H^\sigma(\mathbb R^d)$.
When the parameter $k>0$ is sufficiently small and when $\frac12\le \beta\le 1$, we are able to show that, for any $t\ge 0$,
\ben
E(t) &:=& \|u(t)\|_{H^s(\mathbb R^d)}^2 + \|\tau(t)\|_{H^s(\mathbb R^d)}^2
\notag\\
&& + 2 \int_0^t \left( \eta \|\Lambda^\beta \tau(t')\|_{H^s}^2 + \frac{k}{2} \|\nabla u(t')\|_{H^{s-\beta}}^2 \right) dt'\label{ee0}
\een
obeys
\beq \label{ee}
E(t) \le E(0) + C\, E^{\frac32}(t).
\eeq
A bootstrap argument applied to (\ref{ee}) implies that, if $E(0)$ is sufficiently small, namely
$$
E(0) \le \varepsilon
$$
for some suitable $\varepsilon >0$, then $E(t)$ is bounded uniformly for all time $t>0$, or
$$
E(t) \le C\, \varepsilon,
$$
which allows us to establish the global existence of solutions to (\ref{OBA}). In order to prove the uniqueness, we distinguish between two cases, $\beta =1$ and $\frac12\le \beta<1$. When $\beta=1$, the term $Q(\tau, \na u)$ can be bounded directly. When $\frac12\le \beta<1$, one needs to make use of the wave structure to generate a dissipative term in the velocity field in order to deduct a suitable bound for $Q(\tau, \na u)$.

\vskip .1in
The second part of this paper rigorously assesses that the Oldroyd-B system in (\ref{OBA}) is the vanishing viscosity limit of the Oldroyd-B system with
kinematic dissipation
\beq \label{diss}
\begin{cases}
	\partial_t u+u\cdot\nabla u+\nabla p + \nu (-\Delta)^\alpha u=\nabla\cdot\tau,\ \ \ \ x\in\mathbb{R}^d,\ \ t>0,\\
	\partial_t \tau+u\cdot\nabla\tau+\eta(-\Delta)^\beta\tau+Q(\tau,\nabla u)=D(u),  \\
	\nabla\cdot u=0,\\
	u(0,x)=u_0(x),\ \ \tau(0,x)=\tau_0(x),
\end{cases}
\eeq
where $\nu>0$, $\eta>0$, $0\le \alpha \le 1$ and $\frac12 \le \beta \le 1$.
First of all, (\ref{diss}) always possesses a unique global solution when the initial data is sufficiently small.

\begin{theorem} \label{main1}
Consider (\ref{diss}) with
$$
\nu>0, \quad \eta>0, \quad \frac12 \le \beta\le 1 \quad
\mbox{and}  \quad 0\le \alpha \le \min\{1, 3 \beta -1\}.
$$
Assume $(u_0, \tau_0) \in H^s (\mathbb R^d)$ with $s> 1+\frac{d}{2}$.
There exists small number $\varepsilon>0$ (independent of $\nu$) such that, if
$$
\|(u_0, \tau_0)\|_{H^s} \le \varepsilon,
$$
then (\ref{diss}) has a unique global solution $(u^{(\nu)}, \tau^{(\nu)})$
satisfying
\beno
&& u^{(\nu)} \in C([0, \infty); H^s)\cap L^2(0, \infty; H^{s+\alpha})\cap L^2(0, \infty; H^{s+1-\beta});\\
&& \tau^{(\nu)} \in C([0, \infty); H^s)\cap L^2(0, \infty; H^{s+\beta}).
\eeno
In addition, $(u^{(\nu)}, \tau^{(\nu)})$  admits the following bound
that is uniform in time and in $\nu$,
\beq\label{unibound}
\|(u^{(\nu)}(t), \tau^{(\nu)}(t))\|_{H^s} \le C\, \varepsilon,
\eeq
where $C$ is independent of $t$ and  $\nu$.
\end{theorem}

\vskip .1in
In particular, Theorem \ref{main1} holds for the case when $\alpha=1$ and $\beta=1$, namely the standard Laplacian case. We emphasize that $\varepsilon$ in Theorem \ref{main1} is independent of $\nu$. In addition,  the fact that the bound for the solution $(u^{(\nu)}, \tau^{(\nu)})$ in $H^s$ is uniform in terms of $\nu$ plays crucial role in the proof of the following vanishing viscosity limit.
As $\nu \to 0$, (\ref{diss}) converges to (\ref{OBA}) in the sense as stated
in the following theorem.
\begin{theorem}\label{main2}
Assume 	
	$$
	\nu>0, \quad \eta>0, \quad \frac12 \le \beta\le 1 \quad
	\mbox{and}  \quad 0\le \alpha \le \min\{1, 3 \beta -1\}.
	$$
Let $(u_0, \tau_0) \in H^s (\mathbb R^d)$ with $s> 1+\frac{d}{2}$ and $s\ge 2\alpha + 2\beta -1$. Assume that the norm of $(u_0, \tau_0) \in H^s$ is sufficiently small, namely
$$
\|(u_0, \tau_0)\|_{H^s} \le \varepsilon
$$
such that (\ref{OBA}) and (\ref{diss}) each has a unique global solution. Let $(u, \tau)$ and $(u^{(\nu)}, \tau^{(\nu)})$ be the solutions of
(\ref{OBA}) and (\ref{diss}), respectively. Then,
\beq\label{invi1}
\|(u^{(\nu)}(t), \tau^{(\nu)}(t)) - (u(t), \tau(t))\|_{L^2} \le C\, \nu,
\eeq
where $C$ may depend on $t$ and the initial data but is independent of $\nu$.
\end{theorem}

\vskip .1in
We remark that small data global solutions of (\ref{OBA}) in critical homogeneous Besov spaces have also been obtained (\cite{WuZhao0}).
Due to its special features, the Oldroyd-B model has recently attracted considerable interests from the community of mathematical fluids. A rich array of results have been established on the well-posedness and closely related problems. Interested readers
can consult some of the references listed here, see, e.g., \cite{Bejaoui, J.-Y.Chemin,Chen2018,Miao, Constantin0, Constantin,Constantin-Sun, T.M.F, T.M.J, Fang-Hieber-Zi,D.Y.Fang,E.Fernandez, C.G1, C.G2, La1,M.Hieber,
	Hu-Lelievre, Lin, P.-L.Lions,Z.Lei, Wan, WuZhao0, WuZhao,Ye,Ye-Xu,Zhai2018,zhuyi, R.Z.Zi}. This list is by no means exhaustive.

\vskip .3in
\section{Proof of Theorem \ref{main}}
\setcounter{equation}{0}
	
This section proves Theorem \ref{main}.

\begin{proof}
The proof is naturally divided into two parts. The first part is
for the existence while the second part is for the uniqueness.

\vskip .1in
To prove the global existence of solutions, it suffices to establish the energy inequality in (\ref{ee}) with $E(t)$ being defined in (\ref{ee0}). The proof of (\ref{ee}) is via energy estimates. We need
to separate the homogeneous part of the $H^s$-norm from the
inhomogeneous part. Due to the equivalence of the norm $\|f\|_{H^s}$ with  $\|f\|_{L^2}+\|\Lambda^s f\|_{L^2}$, we combine the $L^2$-part with the homogeneous $\dot{H}^s$-part. Dotting (\ref{OBA}) by $(u, \tau)$ in $L^2$, integrating by parts and making use of $\na\cdot u =0$, we find
\begin{eqnarray}\label{basic-1}
\frac{1}{2}\frac{d}{dt}(\|u\|^2_{L^2}+\|\tau\|^2_{L^2})+\eta\|\Lambda^\beta\tau\|^2_{L^2}
=-(Q(\tau,\nabla u), \tau),
\end{eqnarray}
where $(f, g)$ denotes the inner product in $L^2(\mathbb R^2)$ and we used
\begin{eqnarray*}
\int_{\mathbb R^2} (u\cdot (\nabla\cdot\tau) + D(u)\cdot\tau)\, \mathrm{d}x=0.
\end{eqnarray*}
Applying $\Lambda^s$ to \eqref{OBA} and dotting by $(\Lambda^s u, \Lambda^s \tau)$, we  obtain
\begin{eqnarray}\label{basic-2}
&&\frac{1}{2}\frac{d}{dt}(\|\Lambda^su\|^2_{L^2}+\|\Lambda^s\tau\|^2_{L^2})+\eta\|\Lambda^{s+\beta}\tau\|^2_{L^2}\nonumber\\
&=&-(\Lambda^s(u\cdot\nabla u), \Lambda^su)-(\Lambda^s(u\cdot\nabla \tau), \Lambda^s\tau)-(\Lambda^sQ(\tau,\nabla u), \Lambda^s\tau),
\end{eqnarray}
where we used
\begin{eqnarray*}
\int_{\mathbb R^2} (\Lambda^su\cdot (\Lambda^s\nabla\cdot\tau) + \Lambda^s D(u)\cdot \Lambda^s\tau) \mathrm{d}x=0.
\end{eqnarray*}
We now make use of (\ref{wave1}) to generate a dissipative term on
the velocity field $u$. It is not difficult to check that
\begin{eqnarray}\label{utaue1}
&&\frac{d}{dt}(u,\nabla\cdot\tau)+\frac{1}{2}\|\nabla u\|_{L^2}^2-\|\mathbb{P}\nabla\cdot\tau\|_{L^2}^2\nonumber\\
&=&-((u\cdot\nabla u),\mathbb{P}\nabla\cdot\tau) -(\mathbb{P}\nabla\cdot(u\cdot\nabla\tau),u)-(\mathbb{P}\nabla\cdot Q(\tau,\nabla u), u)\nonumber\\
&&-\eta((-\Delta)^\beta\mathbb{P}\nabla\cdot\tau,u).
\end{eqnarray}
A similar equality also holds for the $\dot{H}^{s-\beta}$ inner product,
\begin{eqnarray}\label{utaue2}
&&\frac{d}{dt}(\Lambda^{s-\beta}u,\Lambda^{s-\beta}\nabla\cdot\tau)+\frac{1}{2}\|\Lambda^{s-\beta}\nabla u\|_{L^2}^2-\|\Lambda^{s-\beta}\mathbb{P}\nabla\cdot\tau\|_{L^2}^2\nonumber\\
&=&-(\Lambda^{s-\beta}(u\cdot\nabla u),\Lambda^{s-\beta}\mathbb{P}\nabla\cdot\tau) -(\Lambda^{s-\beta}\mathbb{P}\nabla\cdot(u\cdot\nabla\tau),\Lambda^{s-\beta}u)\nonumber\\
&&-(\Lambda^{s-\beta}\mathbb{P}\nabla\cdot Q(\tau,\nabla u),\Lambda^{s-\beta} u)
-\eta(\Lambda^{s-\beta}(-\Delta)^\beta\mathbb{P}\nabla\cdot\tau,\Lambda^{s-\beta}u).
\end{eqnarray}
For a constant $k>0$,  \eqref{basic-1}+\eqref{basic-2}+ $k$\eqref{utaue1}+ $k$\eqref{utaue2} leads to
\begin{eqnarray}\label{basicie-1}
&&\frac{1}{2}\frac{d}{dt}(\|u\|^2_{H^s}+\|\tau\|^2_{H^s}+2k(u,\nabla\cdot\tau)_{H^{s-\beta}})+\eta\|\Lambda^\beta\tau\|^2_{H^s}
\nonumber\\
&&+\frac{k}{2}\|\nabla u\|_{H^{s-\beta}}^2-k\|\mathbb{P}\nabla\cdot\tau\|_{H^{s-\beta}}^2=\sum_{i=1}^7I_i,
\end{eqnarray}
where
\begin{eqnarray*}
&& I_1=-k((u\cdot\nabla u),\mathbb{P}\nabla\cdot\tau)_{H^{s-\beta}},\\
&& I_2=-k(\mathbb{P}\nabla\cdot(u\cdot\nabla\tau),u)_{H^{s-\beta}},\\
&& I_3=-k(\mathbb{P}\nabla\cdot Q(\tau,\nabla u),u)_{H^{s-\beta}}, \\
&& I_4=-k\eta(\Lambda^{2\beta}\mathbb{P}\nabla\cdot \tau, u)_{H^{s-\beta}}, \\
&& I_5=-(\Lambda^s(u\cdot\nabla u),\Lambda^su), \\
&&  I_6=-(\Lambda^s(u\cdot\nabla\tau),\Lambda^s\tau), \\
&& I_7=-(Q(\tau,\nabla u),\tau)_{H^s}.
\end{eqnarray*}
Now we estimate $I_1$ through $I_7$. We use the simple facts that $\mathbb P u = u$ if $u$ is divergence-free,  $\mathbb P$ is bounded by $1$ on $H^s(\mathbb R^d)$ and $(\mathbb P f, g) = (f, \mathbb P g)$.
Thanks to $s>1+\frac{d}{2}$, $\frac{1}{2}\leqslant\beta\leqslant1$ and $\nabla\cdot u=0$, we have
\begin{eqnarray*}
|I_1|&\lesssim&\|u\|_{L^\infty}\|\nabla u\|_{L^2}\|\nabla \tau\|_{L^2}+\|\Lambda^{s-\beta+1} \tau\|_{L^2}\| u\|_{L^\infty}\|\Lambda^{s-\beta+1} u\|_{L^2}\\
&\lesssim& \| u\|_{H^s}\|\nabla u\|_{H^{s-\beta}}\|\Lambda^\beta\tau\|_{H^s}.
\end{eqnarray*}
Due to $\frac{1}{2}\leqslant\beta\leqslant1$ and $\nabla\cdot u=0$, we have, by integration by parts,
\begin{eqnarray*}
|I_2|&\lesssim&\| u\|_{L^\infty}\|\nabla u\|_{L^2}\|\nabla \tau\|_{L^2}\\
&& +\|\Lambda^{s-\beta+1} u\|_{L^2}(\|\Lambda^{s-\beta+1} u\|_{L^2}\|\tau\|_{L^\infty}+\|u\|_{L^\infty}\|\Lambda^{s-\beta+1} \tau\|_{L^2})\\
&\lesssim& \| u\|_{H^s}\|\nabla u\|_{H^{s-\beta}}\|\Lambda^{\beta} \tau\|_{H^s}+\|\nabla u\|_{H^{s-\beta}}^2\|\tau \|_{H^s}.
\end{eqnarray*}
Due to $s>1+\frac{d}{2}$ and $\frac{1}{2}\leqslant\beta\leqslant1$, we have, by integration by parts,
\begin{eqnarray*}
|I_3|&\lesssim& \| \nabla u\|_{L^2}^2\| \tau\|_{L^\infty}+\|\Lambda^{s-\beta+1} u\|_{L^2}(\|\Lambda^{s-\beta+1} u\|_{L^2}\|\tau\|_{L^\infty}+\|\nabla u\|_{L^\infty}\|\Lambda^{s-\beta} \tau\|_{L^2})\\
&\lesssim& \| \tau\|_{H^s}\|\nabla u\|_{H^{s-\beta}}^2+\| u\|_{H^s}\|\nabla u\|_{H^{s-\beta}}\|\Lambda^\beta\tau \|_{H^s}.
\end{eqnarray*}
$I_4$ is bounded by
\begin{eqnarray*}
|I_4|\leqslant k\eta\| \Lambda^\beta \tau\|_{H^s}\| \nabla u\|_{H^{s-\beta}}
\leqslant\frac{\eta}{4}\| \Lambda^\beta\tau\|_{H^s}^2+ k^2\eta\| \nabla u\|_{H^{s-\beta}}^2.
\end{eqnarray*}
By $\nabla\cdot u=0$, $\frac{1}{2}\leqslant\beta\leqslant1$ and $s>1+\frac{d}{2}$, we obtain
\begin{eqnarray*}
|I_5|&=&\left|\int(\Lambda^s(u\cdot\nabla u)-u\cdot\nabla\Lambda^s u)\Lambda^s u\mathrm{d}x\right|\\
&\lesssim& \| \Lambda^s u\|_{L^2}^2\| \nabla u\|_{L^\infty}\lesssim \| u\|_{H^s}\| \nabla u\|_{H^{s-\beta}}^2.
\end{eqnarray*}
Similarly,
\begin{eqnarray*}
|I_6|&=&\left|\int(\Lambda^s(u\cdot\nabla \tau )-u\cdot\nabla\Lambda^s \tau)\Lambda^s \tau\mathrm{d}x\right|\\
&\lesssim& \| \Lambda^s \tau\|_{L^2}(\| \nabla u\|_{L^\infty}\| \Lambda^s \tau\|_{L^2}+\| \Lambda^s u\|_{L^2}\| \nabla \tau\|_{L^\infty})\\
&\lesssim& \| u\|_{H^s}\| \Lambda^\beta\tau\|_{H^{s}}^2.
\end{eqnarray*}
Thanks to $\frac{1}{2}\leqslant\beta\leqslant1$, $s>1+\frac{d}{2}$ and $d=2,3$, we have
\begin{eqnarray*}
|I_7|&=& |(Q(\tau,\nabla u),\tau)+(\Lambda^{s-\beta}Q(\tau,\nabla u),\Lambda^{s+\beta}\tau)|\\
&\lesssim& \| \nabla u\|_{L^2}\| \tau\|_{L^4}^2+\| \Lambda^{s+\beta} \tau\|_{L^2}(\| \Lambda^{s-\beta}\nabla u\|_{L^2}\| \tau\|_{L^\infty}+\| \nabla u\|_{L^\infty}\| \Lambda^{s-\beta}\tau\|_{L^2})\\
&\lesssim& \| \nabla u\|_{L^2}\| \tau\|_{L^2}^{2(1-\frac{d}{4})}\| \nabla\tau\|_{L^2}^{\frac{d}{2}}
+\| \tau\|_{H^s}\| \Lambda^{\beta} \tau\|_{H^s}\| \nabla u\|_{H^{s-\beta}}+\| \Lambda^{\beta} \tau\|_{H^s}^2\| u\|_{H^{s}}\\
&\lesssim& \| \tau\|_{H^s}\| \Lambda^{\beta} \tau\|_{H^s}\| \nabla u\|_{H^{s-\beta}}+\| \Lambda^{\beta} \tau\|_{H^s}^2\| u\|_{H^{s}}.
\end{eqnarray*}
In addition, due to $\frac{1}{2}\leqslant\beta\leqslant1$,
$$
k\|\mathbb{P}\nabla\cdot\tau\|_{H^{s-\beta}}^2\leqslant k\|\Lambda^\beta\tau\|^2_{H^s}.
$$
Inserting the estimates for $I_1$ through $I_7$ into \eqref{basicie-1}, we obtain
\begin{eqnarray}\label{basicie-2}
&&\frac{1}{2}\frac{d}{dt}\left(\|u\|^2_{H^s}+\|\tau\|^2_{H^s}+2k(u,\nabla\cdot\tau)_{H^{s-\beta}}\right)\nonumber\\
&& \qquad+ \left(\frac{3}{4}\eta-k\right)\|\Lambda^\beta\tau\|^2_{H^s}
+ \left(\frac{k}{2}-k^2\eta \right)\|\nabla u\|_{H^{s-\beta}}^2\nonumber\\
&\lesssim& \| u\|_{H^s}\|\nabla u\|_{H^{s-\beta}}\|\Lambda^\beta\tau\|_{H^s}+\|\nabla u\|_{H^{s-\beta}}^2\|\tau \|_{H^s}
+ \| u\|_{H^s}\| \nabla u\|_{H^{s-\beta}}^2\nonumber\\
&&+\| u\|_{H^s}\| \Lambda^\beta\tau\|_{H^{s}}^2+\| \tau\|_{H^s}\| \Lambda^{\beta} \tau\|_{H^s}\| \nabla u\|_{H^{s-\beta}}\nonumber\\
&\lesssim&(\| u\|_{H^s}+\| \tau\|_{H^s})(\| \Lambda^\beta\tau\|_{H^{s}}^2+\| \nabla u\|_{H^{s-\beta}}^2).
\end{eqnarray}
By moving $\Lambda^s$ on $u$, and in view of  $\frac{1}{2}\leqslant\beta\leqslant1$, importantly, we have
\begin{eqnarray}
|2k(u,\nabla\cdot\tau)_{H^{s-\beta}}|&\leqslant& 2k\|u\|_{H^s}\|\tau\|_{H^{s+1-2\beta}}\notag \\&\leqslant& 2c_3k\|u\|_{H^s}\|\tau\|_{H^{s}}\notag\\ &\leqslant&\frac{1}{2}\|u\|^2_{H^s}+2c_3^2k^2\|\tau\|^2_{H^s}.\la{cross}
\end{eqnarray}
Choosing $k$ small enough and integrating \eqref{basicie-2} in time and using (\ref{cross}), we have
\begin{eqnarray*}
&&\sup_t\|u\|^2_{H^s}+\sup_t\|\tau\|^2_{H^s}+ 2\int_0^t(\eta \|\Lambda^\beta\tau\|^2_{H^s}+ \frac{k}{2}\|\nabla u\|_{H^{s-\beta}}^2)\mathrm{d}t'\\
&&\lesssim \|u_0\|^2_{H^s}+\|\tau_0\|^2_{H^s}+(\sup_t\|u\|_{H^s}+\sup_t\|\tau\|_{H^s})\int_0^t(\|\Lambda^\beta\tau\|^2_{H^s}+\|\nabla u\|_{H^{s-\beta}}^2)\mathrm{d}t'.
\end{eqnarray*}
Thus, we have established (\ref{ee}). This concludes the proof for the existence part.

\vskip .1in
We now prove the uniqueness. The term $Q(\tau,\nabla u)$ requires special attention. We split the consideration
into two cases: $\beta=1$ and $\frac12\le \beta < 1$. The uniqueness for the case when $\beta=1$ is direct, but the case when $\frac12\le \beta < 1$ is difficult and has to be dealt with by constructing
suitable energy functional.

\vskip .1in
Case 1: $\beta=1$. Assume $(u_1,\tau_1)$ and $(u_2,\tau_2)$ are two solutions of
\eqref{OBA} with the same initial data. Denote $\delta
u=u_1-u_2,\delta\tau=\tau_1-\tau_2$. Then $(\delta u, \delta \tau)$
satisfies
\begin{equation}\label{cha}
\begin{cases}
\partial_t \delta u= \nabla\cdot \delta\tau-u_1\cdot\nabla \delta u-\delta u\cdot\nabla u_2-\nabla \delta P,\\
\partial_t \delta\tau +u_1\cdot\nabla\delta\tau-\eta\Delta\delta \tau=D(\delta u)-\delta u \cdot\nabla\tau_2-Q(\tau_1,\nabla\delta u)-Q(\delta\tau,\nabla u_2),\\
\nabla\cdot\delta u=0,\\
\delta u(x,0)=0;\ \delta\tau(x,0)=0,
\end{cases}
\end{equation}
where $\delta P$ is the corresponding pressure difference.
Taking the $L^2$ inner product of \eqref{cha} with $(\delta u, \delta \tau)$, we have
\begin{eqnarray*}
&&\frac{1}{2}\frac{d}{dt} (\|\delta u\|_{L^2}^2+\|\delta \tau\|_{L^2}^2)+\eta\|\nabla\delta\tau\|_{L^2}^2\\
&=&-\int\delta u\cdot\nabla u_2\cdot\delta u\mathrm{d}x-\int\delta u\cdot\nabla\tau_2\cdot\delta \tau\mathrm{d}x-\int Q(\tau_1,\nabla\delta u)\cdot\delta \tau\mathrm{d}x\\
&&-\int Q(\delta\tau,\nabla u_2)\cdot\delta \tau\mathrm{d}x
+ \int (\delta u\cdot(\nabla\cdot\delta\tau) + D(\delta u)\cdot\delta\tau)\mathrm{d}x\\
&\leqslant& \|\nabla u_2\|_{L^\infty}\|\delta u\|_{L^2}^2+\|\nabla\tau_2\|_{L^\infty}\|\delta u\|_{L^2}\|\delta \tau\|_{L^2}\\
&& + \,c\,(\|\tau_1\|_{L^\infty}\|\nabla\delta \tau\|_{L^2}+\|\nabla\tau_1\|_{L^\infty}\|\delta \tau\|_{L^2})\|\delta u\|_{L^2}+\|\nabla u_2\|_{L^\infty}\|\delta \tau\|_{L^2}^2\\
&\leqslant& c\,(\|\nabla u_2\|_{L^\infty}+\|\nabla \tau_2\|_{L^\infty}+\|\nabla \tau_1\|_{L^\infty}+\| \tau_1\|_{L^\infty}^2)(\|\delta u\|_{L^2}^2+\|\delta \tau\|_{L^2}^2)+\frac{\eta}{2}\|\nabla\delta\tau\|_{L^2}^2,
\end{eqnarray*}
where we have used the fact that
$$
\int (\delta u\cdot(\nabla\cdot\delta\tau) + D(\delta u)\cdot\delta\tau)\mathrm{d}x =0.
$$
It then follows from Gronwall's inequality that $\delta u=\delta\tau=0$.

\vskip .1in
Case 2: $\frac{1}{2}\leqslant\beta<1$. Assume $(u_1,\tau_1)$ and $(u_2,\tau_2)$ are two solutions of
\eqref{OBA} with the same initial data. Denote $\delta
u=u_1-u_2,\delta\tau=\tau_1-\tau_2$. Then $(\delta u, \delta \tau)$
satisfies
\begin{equation}\label{cha1}
\begin{cases}
\partial_t \delta u = \nabla\cdot \delta\tau-u_1\cdot\nabla \delta u-\delta u\cdot\nabla u_2+\nabla \delta P,\\
\partial_t \delta\tau+u_1\cdot\nabla\delta\tau+\eta\Lambda^{2\beta}\delta \tau=D(\delta u)-\delta u \cdot\nabla\tau_2-Q(\tau_1,\nabla\delta u)-Q(\delta\tau,\nabla u_2),\\
\nabla\cdot\delta u=0,\\
\delta u(x,0)=0;\ \delta\tau(x,0)=0.
\end{cases}
\end{equation}
Dotting (\ref{cha1}) by $(\delta u, \delta\tau)$ yields
\begin{eqnarray}\label{basicuniqueness-1}
\hskip -.3in &&  \frac{1}{2}\frac{d}{dt}(\|\delta u\|^2_{L^2}+\|\delta\tau\|^2_{L^2})+\eta\|\Lambda^\beta\delta\tau\|^2_{L^2}\nonumber\\
\hskip -.3in &=&  -(\delta u\cdot\nabla u_2,\delta u)-(\delta u\cdot\nabla \tau_2,\delta \tau)
-(Q(\tau_1,\nabla\delta u), \delta\tau)-(Q(\delta\tau,\nabla u_2), \delta\tau).
\end{eqnarray}
Applying $\Lambda^\beta$ to \eqref{cha1} and then dotting by $(\Lambda^\beta\delta u, \Lambda^\beta\delta \tau)$ lead to
\begin{eqnarray}\label{basicuniqueness-2}
\hskip -.3in &&\frac{1}{2}\frac{d}{dt}(\|\Lambda^\beta \delta u\|^2_{L^2}+\|\Lambda^\beta\delta\tau\|^2_{L^2})+\eta\|\Lambda^{2\beta}\delta\tau\|^2_{L^2}\nonumber\\
\hskip -.3in &=& -(\Lambda^\beta(u_1\cdot\nabla\delta u), \Lambda^\beta\delta u)-(\Lambda^\beta(\delta u\cdot\nabla u_2), \Lambda^\beta\delta u)-(\Lambda^\beta(u_1\cdot\nabla \delta\tau), \Lambda^\beta\delta\tau)\nonumber\\
\hskip -.3in &&-(\Lambda^\beta(\delta u\cdot\nabla \tau_2), \Lambda^\beta\delta\tau)-(\Lambda^\beta Q(\tau_1,\nabla\delta u), \Lambda^\beta\delta\tau)-(\Lambda^\beta Q(\delta\tau,\nabla u_2), \Lambda^\beta\delta\tau).
\end{eqnarray}
Applying $\mathbb{P}\nabla\cdot $ to the second equation of \eqref{cha1}, we have
\begin{eqnarray}\label{OBAuniquess-2}
&&\partial_ t\mathbb{P}\nabla\cdot\delta\tau+\mathbb{P}\nabla\cdot(u_1\cdot\nabla\delta\tau)+\eta\Lambda^{2\beta}\mathbb{P}\nabla\cdot\delta\tau\nonumber\\
&&=\frac{1}{2}\Delta\delta u-\mathbb{P}\nabla\cdot(\delta u\cdot\nabla\tau_2)-\mathbb{P}\nabla\cdot Q(\tau_1,\nabla\delta u)-\mathbb{P}\nabla\cdot Q(\delta\tau,\nabla u_2).
\end{eqnarray}
Taking the $L^2$ inner product of the first equation of \eqref{cha1} with $\mathbb{P}\nabla\cdot\delta\tau$ and the $L^2$ inner product of \eqref{OBAuniquess-2} with $\delta u$ separately, we have
\begin{eqnarray}\label{utaue1uniqueness}
&&\frac{d}{dt}(\delta u,\nabla\cdot\delta\tau)+\frac{1}{2}\|\nabla \delta u\|_{L^2}^2-\|\mathbb{P}\nabla\cdot\delta\tau\|_{L^2}^2\nonumber\\
&=&-((u_1\cdot\nabla\delta u),\mathbb{P}\nabla\cdot\delta\tau)-((\delta u\cdot\nabla u_2),\mathbb{P}\nabla\cdot\delta\tau) -(\mathbb{P}\nabla\cdot(u_1\cdot\nabla\delta\tau),\delta u)\nonumber\\
&&-(\mathbb{P}\nabla\cdot(\delta u\cdot\nabla\tau_2),\delta u)-(\mathbb{P}\nabla\cdot Q(\tau_1,\nabla\delta u), \delta u)-(\mathbb{P}\nabla\cdot Q(\delta\tau,\nabla u_2), \delta u)\nonumber\\
&&-\eta(\Lambda^{2\beta}\mathbb{P}\nabla\cdot\delta\tau,\delta u).
\end{eqnarray}
For a  positive constant $k_1$ to be determined later,  $\eqref{basicuniqueness-1}+\eqref{basicuniqueness-2}+k_1\eqref{utaue1uniqueness}$ gives
\begin{eqnarray}\label{basicie-11}
&&\frac{1}{2}\frac{d}{dt}(\|\delta u\|^2_{H^\beta}+\|\delta\tau\|^2_{H^\beta}+2k_1(\delta u,\nabla\cdot\delta\tau))+\eta\|\Lambda^\beta\delta\tau\|^2_{H^\beta}
\nonumber\\
&&+\frac{k_1}{2}\|\nabla\delta u\|_{L^2}^2-k_1\|\mathbb{P}\nabla\cdot\delta\tau\|_{L^2}^2=\sum_{i=1}^7I'_i,
\end{eqnarray}
where
\begin{eqnarray*}
&& I'_1=-k_1((u_1\cdot\nabla\delta u),\mathbb{P}\nabla\cdot\delta\tau)-k_1((\delta u\cdot\nabla u_2),\mathbb{P}\nabla\cdot\delta\tau), \\
&& I'_2=-k_1(\mathbb{P}\nabla\cdot(u_1\cdot\nabla\delta\tau),\delta u)
-k_1(\mathbb{P}\nabla\cdot(\delta u\cdot\nabla\tau_2),\delta u),\\
&& I'_3=-k_1(\mathbb{P}\nabla\cdot Q(\tau_1,\nabla\delta u), \delta u)-k_1(\mathbb{P}\nabla\cdot Q(\delta\tau,\nabla u_2), \delta u),\\
&& I'_4=-k_1\eta(\Lambda^{2\beta}\mathbb{P}\nabla\cdot\delta\tau,\delta u), \\
&& I'_5=-(\Lambda^\beta(u_1\cdot\nabla\delta u), \Lambda^\beta\delta u)-((\delta u\cdot\nabla u_2), \delta u)_{H^\beta}, \\
&& I'_6=-(\Lambda^\beta(u_1\cdot\nabla \delta\tau), \Lambda^\beta\delta\tau)-((\delta u\cdot\nabla \tau_2), \delta\tau)_{H^\beta},\\
&& I'_7=-( Q(\tau_1,\nabla\delta u), \delta\tau)_{H^\beta}-( Q(\delta\tau,\nabla u_2), \delta\tau)_{H^\beta}.
\end{eqnarray*}
By H\"{o}lder's and Sobolev's inequalities,
\begin{eqnarray*}
&&|I'_1|\lesssim \|u_1\|_{L^\infty}\|\nabla\delta u\|_{L^2}\|\nabla \delta\tau\|_{L^2}+\|\nabla u_2\|_{L^\infty}\|\delta u\|_{L^2}\|\nabla \delta\tau\|_{L^2}, \\
&&|I'_2|\lesssim \|u_1\|_{L^\infty}\|\nabla\delta u\|_{L^2}\|\nabla \delta\tau\|_{L^2}+\|\nabla \tau_2\|_{L^\infty}\|\nabla\delta u\|_{L^2}\| \delta u\|_{L^2},\\
&&|I'_3|\lesssim\|\tau_1\|_{L^\infty}\|\nabla\delta u\|_{L^2}^2+\|\nabla u_2\|_{L^\infty}\|\nabla\delta u\|_{L^2}\| \delta\tau\|_{L^2},\\
&&
|I'_4|\leqslant\frac{\eta}{4}\|\Lambda^{2\beta}\delta\tau\|^2_{L^2}+k_1^2\eta\|\nabla\delta u\|_{L^2}^2.
\end{eqnarray*}
Since $\na\cdot u_1 =0$, $I_5'$ can be written as
\begin{eqnarray*}
I'_5=-(\Lambda^\beta(u_1\cdot\nabla\delta u) -u_1 \cdot \na \Lambda^\beta\delta u, \Lambda^\beta\delta u)-((\delta u\cdot\nabla u_2), \delta u)_{H^\beta}.
\end{eqnarray*}
By a standard commutator estimate,
\begin{eqnarray*}
|I'_5| &\lesssim& \|\nabla u_1\|_{L^{\frac{d}{2-2\beta}}}\|\Lambda^\beta\delta u\|_{L^{\frac{2d}{d-2+2\beta}}}^2+\|\Lambda^\beta u_1\|_{L^{\frac{d}{1-\beta}}}\|\nabla \delta u\|_{L^2}\|\Lambda^\beta\delta u\|_{L^{\frac{2d}{d-2+2\beta}}}
\\
&&+\|\nabla u_2\|_{L^\infty}\|\delta u\|_{L^2}^2+\|\nabla u_2\|_{L^\infty}\|\Lambda^\beta\delta u\|_{L^2}^2+\|\delta u\|_{L^{\frac{2d}{d-2\beta}}}\|\Lambda^\beta \nabla u_2\|_{L^{\frac{d}{\beta}}}\|\Lambda^\beta\delta u\|_{L^2}\nonumber\\
&\lesssim&(\|\nabla u_1\|_{L^{\frac{d}{2-2\beta}}}+\|\Lambda^\beta u_1\|_{L^{\frac{d}{1-\beta}}})\|\nabla\delta u\|_{L^2}^2
+\|\nabla u_2\|_{L^\infty}\|\delta u\|_{L^2}^2\\
&&+(\|\nabla u_2\|_{L^\infty}+\|\Lambda^\beta \nabla u_2\|_{L^{\frac{d}{\beta}}})\|\Lambda^\beta\delta u\|_{L^2}^2.
\end{eqnarray*}
By H\"{o}lder's inequality,
\begin{eqnarray*}
|I'_6|&\lesssim&\| u_1\|_{L^\infty}\|\nabla\delta \tau\|_{L^2}\|\Lambda^{2\beta}\delta \tau\|_{L^2}\\
&&
+ \|\nabla \tau_2\|_{L^\infty}\|\delta u\|_{L^2}(\|\delta \tau\|_{L^2}
 +\|\Lambda^{2\beta}\delta \tau\|_{L^2}),\\
|I'_7|&\lesssim&\| \tau_1\|_{L^\infty}\|\nabla\delta u\|_{L^2}(\|\delta \tau\|_{L^2}+\|\Lambda^{2\beta}\delta \tau\|_{L^2})\\
&& +\|\nabla u_2\|_{L^\infty}\|\delta \tau\|_{L^2}(\|\delta \tau\|_{L^2}+\|\Lambda^{2\beta}\delta \tau\|_{L^2}).
\end{eqnarray*}
We insert the estimates above for $I'_1$ through $I'_7$ in \eqref{basicie-11}. If the initial data is small enough, namely
\begin{eqnarray*}
\|u_0\|_{H^s}+\|\tau_0\|_{H^s}\leqslant \varepsilon
\end{eqnarray*}
for sufficiently small $\varepsilon>0$, we can choose $k_1$ and $t$ small enough to obtain the desired uniqueness.
This completes the proof of Theorem \ref{main}.
\end{proof}

\vskip .3in
\section{Proof of Theorems \ref{main1} and  \ref{main2}}
\setcounter{equation}{0}

This section proves Theorems \ref{main1} and \ref{main2}.

\begin{proof}[Proof of Theorem \ref{main1}]
The proof of Theorem \ref{main1} is very close to that for Theorem \ref{main}. We shall omit most of the details but to point out the differences. The differences are due to the extra term
$\nu (-\Delta)^\alpha u^{(\nu)}$. (\ref{basicie-1}) would now contain
two extra terms and is given by
\begin{eqnarray*}
&&\frac{1}{2}\frac{d}{dt}(\|u^{(\nu)}\|^2_{H^s}+\|\tau^{(\nu)}\|^2_{H^s}
+2k(u^{(\nu)},\nabla\cdot\tau^{(\nu)})_{H^{s-\beta}})+\eta\|\Lambda^\beta\tau^{(\nu)}\|^2_{H^s} + \nu \|\Lambda^\alpha u^{(\nu)}\|^2_{H^s}
\nonumber\\
&&+\frac{k}{2}\|\nabla u^{(\nu)}\|_{H^{s-\beta}}^2-k\|\mathbb{P}\nabla\cdot\tau^{(\nu)}\|_{H^{s-\beta}}^2=\sum_{i=1}^8I_i,
\end{eqnarray*}
where $I_1$ through $I_7$ are the same as before, and $I_8$ is given by
$$
I_8 = \nu \,k\, ((-\Delta)^\alpha u^{(\nu)}, \nabla\cdot\tau^{(\nu)})_{H^{s-\beta}}.
$$
The estimates for $I_1$ through $I_7$ are the same as before and $I_8$ can be bounded by
$$
|I_8| \le \nu k\|\Lambda^{2\alpha -3\beta +1} u^{(\nu)}\|_{H^s} \, \|\Lambda^\beta \tau^{(\nu)}\|_{H^s}.
$$	
When $\alpha \le  \min\{1, 3 \beta -1\}$, we have $2\alpha -3\beta +1 \le \alpha$ and
$$
|I_8| \le \frac{\nu}{2} \|\Lambda^{\alpha} u^{(\nu)}\|^2_{H^s} + \frac{\nu k^2}{2} \|\Lambda^\beta \tau^{(\nu)}\|^2_{H^s}.
$$
The rest of the proof is almost identical to that for Theorem \ref{main}. The crucial fact that the bound for $(u^{(\nu)},\tau^{(\nu)})$ in $H^s$ obtained from this process is uniform in $\nu$.  We omit further details.
\end{proof}

\vskip .1in
We now turn to the proof of Theorem \ref{main2}.
\begin{proof}[Proof of Theorem \ref{main2}]
	We distinguish between two cases: Case I: $\beta=1$ and Case II: $\frac12 \le \beta< 1$. The first case is relatively easy while the second case is more delicate.  The fact that the bound for the solution $(u^{(\nu)}, \tau^{(\nu)})$ in $H^s$ is uniform in terms of $\nu$ plays crucial role in the proof.
	
	\vskip .1in
	Case I: $\beta=1$. The difference $(\delta u, \delta \tau)$ with
	$$
	\delta
	u=u^{(\nu)}-u,\quad \delta\tau=\tau^{(\nu)}-\tau
	$$
	satisfies
	\begin{equation}\label{good}
	\begin{cases}
	\partial_t \delta u + u^{(\nu)}\cdot\na \delta u +  \nu \,(-\Delta)^\alpha \delta u=-\nu \,(-\Delta)^\alpha u +  \nabla\cdot \delta\tau-\delta u\cdot\nabla u-\nabla \delta P,\\
	\partial_t \delta\tau +u^{(\nu)}\cdot\nabla\delta\tau-\eta\Delta\delta \tau=D(\delta u)-\delta u \cdot\nabla\tau-Q(\tau, \na\delta u) -Q(\delta \tau,\nabla u^{(\nu)}),\\
	\nabla\cdot\delta u=0,\\
	\delta u(x,0)=0;\ \delta\tau(x,0)=0,
	\end{cases}
	\end{equation}
	where $\delta P$ is the corresponding pressure difference.
	Taking the $L^2$ inner product of \eqref{good} with $(\delta u, \delta \tau)$, we have
	\begin{eqnarray*}
		&&\frac{1}{2}\frac{d}{dt} (\|\delta u\|_{L^2}^2+\|\delta \tau\|_{L^2}^2)+\nu \|\Lambda^\alpha \delta u\|_{L^2}^2 +  \eta\|\nabla\delta\tau\|_{L^2}^2\\
		&=&-\nu\,\int(-\Delta)^\alpha u\cdot \delta u\,dx -\int\delta u\cdot\nabla u\cdot\delta u \,\mathrm{d}x-\int\delta u\cdot\nabla\tau\cdot\delta \tau\,\mathrm{d}x\\
		&&-\int Q(\tau,\nabla\delta u)\cdot\delta \tau\,\mathrm{d}x
		-\int Q(\delta\tau,\nabla u^{(\nu)})\cdot\delta \tau\,\mathrm{d}x
		\\
		&\leqslant& \nu \|u\|_{H^{2\alpha}}\, \|\delta u\|_{L^2} +
		\|\nabla u\|_{L^\infty}\|\delta u\|_{L^2}^2+\|\nabla\tau\|_{L^\infty}\|\delta u\|_{L^2}\|\delta \tau\|_{L^2}\\
		&& + \,c\,(\|\tau\|_{L^\infty}\|\nabla\delta \tau\|_{L^2}+\|\nabla\tau\|_{L^\infty}\|\delta \tau\|_{L^2})\|\delta u\|_{L^2}+\|\nabla u^{(\nu)}\|_{L^\infty}\|\delta \tau\|_{L^2}^2\\
		&\leqslant& \nu^2 \|u\|_{H^{2\alpha}}^2 + \frac{\eta}{2}\|\nabla\delta\tau\|_{L^2}^2\\
		&& + C\, (1 + \| u\|_{H^s}+\|\tau\|_{H^s}+\|u^{(\nu)}\|_{H^s}+\| \tau\|_{H^{s-1}}^2)(\|\delta u\|_{L^2}^2+\|\delta \tau\|_{L^2}^2).
	\end{eqnarray*}
	Here we have used the fact that
	$$
	\int (\delta u\cdot(\nabla\cdot\delta\tau) + D(\delta u)\cdot\delta\tau)\mathrm{d}x =0.
	$$
	(\ref{invi1}) then follows from Gronwall's inequality and the uniform bound (in $\nu$) for $\|\tau^{(\nu)}\|_{H^s}$.
	
	\vskip .1in
	Case 2: $\frac{1}{2}\leqslant\beta<1$. The difference $(\delta u, \delta \tau)$
	satisfies
	\begin{equation}\label{good2}
	\begin{cases}
	\partial_t \delta u + u^{(\nu)}\cdot\na \delta u +  \nu \,(-\Delta)^\alpha \delta u=-\nu \,(-\Delta)^\alpha u +  \nabla\cdot \delta\tau-\delta u\cdot\nabla u-\nabla \delta P,\\
	\partial_t \delta\tau +u^{(\nu)}\cdot\nabla\delta\tau + \eta(-\Delta)^\beta
	\delta \tau=D(\delta u)-\delta u \cdot\nabla\tau-Q(\tau, \na\delta u) -Q(\delta \tau,\nabla u^{(\nu)}),\\
	\nabla\cdot\delta u=0,\\
	\delta u(x,0)=0;\ \delta\tau(x,0)=0.
	\end{cases}
	\end{equation}

Dotting (\ref{good2}) by $(\delta u, \delta\tau)$ yields
    \begin{eqnarray*}
	\hskip -.3in &&\frac{1}{2}\frac{d}{dt}(\| \delta u\|^2_{L^2}+\|\delta\tau\|^2_{L^2})+\eta\|\Lambda^{\beta}\delta\tau\|^2_{L^2} + \nu \|\Lambda^{\alpha} \delta u\|^2_{L^2}\nonumber\\
	\hskip -.3in &=& - \nu (\Lambda^{2\alpha } u,  \delta u) -(\delta u\cdot\nabla u, \delta u)-(\delta u\cdot\nabla \tau, \delta\tau)\nonumber\\
	\hskip -.3in &&-(Q(\tau,\nabla\delta u), \delta\tau)-( Q(\delta\tau,\nabla u^{(\nu)}), \delta\tau).
	\end{eqnarray*}

	Applying $\Lambda^\beta$ to \eqref{good2} and then dotting by $(\Lambda^\beta\delta u, \Lambda^\beta\delta \tau)$ lead to
	\begin{eqnarray}\label{jj8}
	\hskip -.3in &&\frac{1}{2}\frac{d}{dt}(\|\Lambda^\beta \delta u\|^2_{L^2}+\|\Lambda^\beta\delta\tau\|^2_{L^2})+\eta\|\Lambda^{2\beta}\delta\tau\|^2_{L^2} + \nu \|\Lambda^{\alpha+ \beta} \delta u\|^2_{L^2}\nonumber\\
	\hskip -.3in &=& - \nu (\Lambda^{2\alpha +\beta} u, \Lambda^\beta \delta u) -(\Lambda^\beta(u^{(\nu)}\cdot\nabla\delta u), \Lambda^\beta\delta u)-(\Lambda^\beta(\delta u\cdot\nabla u), \Lambda^\beta\delta u)\nonumber\\
	\hskip -.3in &&-(\Lambda^\beta(u^{(\nu)}\cdot\nabla \delta\tau), \Lambda^\beta\delta\tau)-(\Lambda^\beta(\delta u\cdot\nabla \tau), \Lambda^\beta\delta\tau)-(\Lambda^\beta Q(\tau,\nabla\delta u), \Lambda^\beta\delta\tau)\nonumber\\
	\hskip -.3in &&-(\Lambda^\beta Q(\delta\tau,\nabla u^{(\nu)}), \Lambda^\beta\delta\tau).
	\end{eqnarray}

	Applying $\mathbb{P}\nabla\cdot $ to the second equation of \eqref{good2}, we have
	\begin{eqnarray}\label{jj9}
	&&\partial_ t\mathbb{P}\nabla\cdot\delta\tau+\mathbb{P}\nabla\cdot(u^{(\nu)}\cdot\nabla\delta\tau)+\eta\Lambda^{2\beta}\mathbb{P}\nabla\cdot\delta\tau\nonumber\\
	&&=\frac{1}{2}\Delta\delta u-\mathbb{P}\nabla\cdot(\delta u\cdot\nabla\tau)-\mathbb{P}\nabla\cdot Q(\tau,\nabla\delta u)-\mathbb{P}\nabla\cdot Q(\delta\tau,\nabla u^{(\nu)}).
	\end{eqnarray}
	Taking the $L^2$ inner product of the first equation of \eqref{good2} with $\mathbb{P}\nabla\cdot\delta\tau$ and the $L^2$ inner product of \eqref{jj9} with $\delta u$, we have
	\begin{eqnarray}\label{jj11}
	&&\frac{d}{dt}(\delta u,\nabla\cdot\delta\tau)+\frac{1}{2}\|\nabla \delta u\|_{L^2}^2-\|\mathbb{P}\nabla\cdot\delta\tau\|_{L^2}^2\nonumber\\
	&=&-\nu ((-\Delta)^\alpha  u^{(\nu)}, \nabla\cdot\delta\tau) -((u^{(\nu)}\cdot\nabla\delta u),\mathbb{P}\nabla\cdot\delta\tau)-((\delta u\cdot\nabla u),\mathbb{P}\nabla\cdot\delta\tau)\nonumber\\
	&& -(\mathbb{P}\nabla\cdot(u^{(\nu)}\cdot\nabla\delta\tau),\delta u)-(\mathbb{P}\nabla\cdot(\delta u\cdot\nabla\tau),\delta u)-(\mathbb{P}\nabla\cdot Q(\tau,\nabla\delta u), \delta u)\nonumber\\
	&&-(\mathbb{P}\nabla\cdot Q(\delta\tau,\nabla u^{(\nu)}), \delta u)-\eta(\Lambda^{2\beta}\mathbb{P}\nabla\cdot\delta\tau,\delta u).
	\end{eqnarray}
	We choose a  positive constant $k_3$ satisfying, for a suitable constant $C>0$,
	$$
	0<k_3 \le C\, \min\{1, \eta\}.
	$$
	Then $\eqref{jj8}+k_3\eqref{jj11}$ gives
	\begin{eqnarray}\label{jj10}
	&&\frac{1}{2}\frac{d}{dt}(\|\delta u\|^2_{H^\beta}+\|\delta\tau\|^2_{H^\beta}+2k_3(\delta u,\nabla\cdot\delta\tau))+\eta\|\Lambda^\beta\delta\tau\|^2_{H^\beta}
	\nonumber\\
	&& + \nu \|\Lambda^{\alpha} \delta u\|^2_{H^\beta} +\frac{k_3}{2} \|\nabla\delta u\|_{L^2}^2-k_3\|\mathbb{P}\nabla\cdot\delta\tau\|_{L^2}^2
	=\sum_{i=1}^{20} K_i,
	\end{eqnarray}
	where
\begin{eqnarray*}
	&& K_1 = - \nu (\Lambda^{2\alpha +\beta} u, \Lambda^\beta \delta u), \qquad  K_2 = -(\Lambda^\beta(u^{(\nu)}\cdot\nabla\delta u), \Lambda^\beta\delta u),\\
	&&
	K_3 = -(\Lambda^\beta(\delta u\cdot\nabla u), \Lambda^\beta\delta u), \qquad K_4 = -(\Lambda^\beta(u^{(\nu)}\cdot\nabla \delta\tau), \Lambda^\beta\delta\tau), \\
	&&
	K_5 = -(\Lambda^\beta(\delta u\cdot\nabla \tau), \Lambda^\beta\delta\tau), \qquad
	K_6 = -(\Lambda^\beta Q(\tau,\nabla\delta u), \Lambda^\beta\delta\tau),\\
	&& K_7 =
	-(\Lambda^\beta Q(\delta\tau,\nabla u^{(\nu)}), \Lambda^\beta\delta\tau), \qquad
	K_8 = -k_3\,\nu ((-\Delta)^\alpha  u^{(\nu)}, \nabla\cdot\delta\tau),\\
	&& K_9 = -k_3 ((u^{(\nu)}\cdot\nabla\delta u),\mathbb{P}\nabla\cdot\delta\tau), \qquad
	K_{10} = - k_3 ((\delta u\cdot\nabla u),\mathbb{P}\nabla\cdot\delta\tau),\\
	&& K_{11} = -k_3 (\mathbb{P}\nabla\cdot(u^{(\nu)}\cdot\nabla\delta\tau),\delta u), \qquad K_{12} = - k_3 (\mathbb{P}\nabla\cdot(\delta u\cdot\nabla\tau),\delta u),\\
	&& K_{13} = - k_3 (\mathbb{P}\nabla\cdot Q(\tau,\nabla\delta u), \delta u), \qquad  K_{14} = -k_3 (\mathbb{P}\nabla\cdot Q(\delta\tau,\nabla u^{(\nu)}), \delta u),\\
	&& K_{15} =-k_3 \eta(\Lambda^{2\beta}\mathbb{P}\nabla\cdot\delta\tau,\delta u),\qquad
       K_{16}=- \nu (\Lambda^{2\alpha } u,  \delta u),\\ 
    && K_{17}=-(\delta u\cdot\nabla u, \delta u),\qquad
       K_{18}=-(\delta u\cdot\nabla \tau, \delta\tau),\\
	&& K_{19}=-(Q(\tau,\nabla\delta u), \delta\tau),\qquad
       K_{20}-( Q(\delta\tau,\nabla u^{(\nu)}), \delta\tau).
	\end{eqnarray*}

The terms above can be bounded as follows. All the constants in the estimates are independent of $\nu$. By H\"{o}lder's inequality,
\beno
|K_1| \le \nu ^2 \|\Lambda^{2\alpha +\beta} u\|_{L^2}^2 + C\, \|\delta u\|_{H^\beta}^2.
\eeno
Due to $\na\cdot u^{(\nu)}=0$ and by a standard commutator estimate,
\beno
|K_2| &\le& C\, \|u^{(\nu)}\|_{H^s} \,\|\delta u\|_{H^\beta}^2 + C\, \|u^{(\nu)}\|_{H^s} \,\|\na \delta u\|_{L^2}
\|\Lambda^\beta \delta u\|_{L^2} \\
&\le& \frac{k_3}{16} \|\nabla\delta u\|_{L^2}^2 + C\, (1+ k_3^{-1}\|u^{(\nu)}\|_{H^s}) \, \|u^{(\nu)}\|_{H^s} \,\|\delta u\|_{H^\beta}^2.
\eeno
Clearly, for $q_1$ and $q_2$ satisfying $\frac1{q_1} =\frac12 - \frac{\beta}{d}$ and $\frac1{q_2} = \frac12 -\frac1{q_1}$,
\beno
|K_3| &\le& C\, \|u\|_{H^s} \, \|\delta u\|_{H^\beta}^2
+ \|\delta u\|_{L^{q_1}} \,\|\Lambda^\beta \na u\|_{L^{q_2}}\, \|\Lambda^\beta \delta u\|_{L^2}\\
&\le& C\,  \|u\|_{H^s} \, \|\delta u\|_{H^\beta}^2.
\eeno
By a commutator estimate,
\beno
|K_4| &\le& C\, \|\Lambda^\beta u^{(\nu)}\|_{L^\infty}\, \|\na \delta \tau\|_{L^2}\, \|\Lambda^\beta \delta \tau\|_{L^2} + C\,\|\na u^{(\nu)}\|_{L^\infty} \|\Lambda^\beta \delta \tau\|_{L^2}^2\\
&\le& \frac{\eta}{16} \|\Lambda^\beta\delta\tau\|^2_{H^\beta}
+ C\, (\eta^{-1}\|u^{(\nu)}\|_{H^s}^2 + \|u^{(\nu)}\|_{H^s}) \|\delta \tau\|_{H^\beta}^2.
\eeno
$K_5$ can be similarly bounded as $K_3$,
$$
|K_5| \le C\, \|\tau\|_{H^s} \, (\|\delta u\|_{H^\beta}^2 + \|\delta \tau\|_{H^\beta}^2).
$$
By H\"{o}lder's inequality,
\beno
|K_6| &\le& \|\Lambda^{2\beta} \delta\tau\|_{L^2} \, \|\tau\|_{L^\infty}\,
\|\na \delta u\|_{L^2} \\
&\le& \frac{\eta}{16} \|\Lambda^\beta\delta\tau\|^2_{H^\beta} + C \, \eta^{-1} \|\tau\|^2_{H^{s-1}}\, \|\na \delta u\|_{L^2}^2\\
&\le& \frac{\eta}{16} \|\Lambda^\beta\delta\tau\|^2_{H^\beta} + \frac{k_3}{16} \|\na \delta u\|_{L^2}^2,
\eeno
where we have used the smallness of the solution
$$
C\, \eta^{-1} \|\tau\|_{H^{s-1}}^2 \le C\, \eta^{-1}\, \varepsilon^2 \le \frac{k_3}{16}.
$$
By H\"{o}lder's inequality,
\beno
|K_7| &\le& \|\Lambda^{2\beta} \delta \tau\|_{L^2} \,\|\na u^{(\nu)}\|_{L^\infty}\, \|\delta \tau \|_{L^2}\\
 &\le& \frac{\eta}{16} \|\Lambda^\beta\delta\tau\|^2_{H^\beta} + C \, \eta^{-1} \|u^{(\nu)}\|^2_{H^s}\, \|\delta \tau \|_{L^2}^2.
\eeno
Clearly,
\beno
|K_8| &\le&  k_3 \nu \|\Lambda^{2\alpha + 1 -\beta} u^{(\nu)}\|_{L^2} \,
\|\Lambda^\beta \delta \tau\|_{L^2}\\
&\le& \nu^2  \|\Lambda^{2\alpha + 1 -\beta} u^{(\nu)}\|_{L^2}^2 + C\, k_3^2\,
\| \delta \tau\|_{H^\beta}^2.
\eeno
$K_9$ can be similarly handled as $K_6$,
\beno
|K_9| &\le&  k_3 \,\|u^{(\nu)} \|_{L^\infty}\, \|\na \delta u\|_{L^2}\,
\|\na \delta \tau\|_{L^2} \\
&\le& \frac{\eta}{16} \|\Lambda^\beta\delta\tau\|^2_{H^\beta} + C\, \eta^{-1}
k_3^2\, \|u^{(\nu)}\|^2_{H^{s-1}}\, \|\na \delta u\|^2_{L^2}\\
&\le&  \frac{\eta}{16} \|\Lambda^\beta\delta\tau\|^2_{H^\beta} + \frac{k_3}{16} \|\na \delta u\|_{L^2}^2,
\eeno
where we have used the smallness of the solution
$$
C\, k_3 \eta^{-1} \|u^{(\nu)}\|_{H^{s-1}}^2 \le C\, k_3 \eta^{-1}\, \varepsilon^2 \le \frac{1}{16}.
$$
We emphasize that $\|u^{(\nu)}\|_{H^{s}} \le C\, \varepsilon$ with $\varepsilon$ independent of $\nu$, as stated in Theorem \ref{main1}. For $\beta\ge \frac12$,
\beno
|K_{10}| &\le&  k_3 \, \|\delta u\|_{L^2} \,\|\na u\|_{L^\infty}\, \|\na\delta \tau\|_{L^2} \\
&\le& \frac{\eta}{16} \|\Lambda^\beta\delta\tau\|^2_{H^\beta} + C\, k_3^2\, \|u\|_{H^s}^2\, \|\delta u\|_{L^2}^2.
\eeno
$K_{11}$ admits the same bound as $K_9$,
\beno
|K_{11}| &\le&  k_3\, \|u^{(\nu)}\|_{L^\infty}\,  \|\na \delta u\|_{L^2}\,
\|\na \delta \tau\|_{L^2} \\
&\le&  \frac{\eta}{16} \|\Lambda^\beta\delta\tau\|^2_{H^\beta} + \frac{k_3}{16} \|\na \delta u\|_{L^2}^2.
\eeno
$K_{12}$ can be bounded directly,
\beno
|K_{12}| &\le&  k_3\, \|\delta u\|_{L^2}\,\|\na \tau\|_{L^\infty}\, \|\na \delta u\|_{L^2} \\
 &\le& \frac{k_3}{16} \|\na \delta u\|_{L^2}^2 + C\, k_3 \|\tau\|_{H^s}^2\,\|\delta u\|_{L^2}^2.
\eeno
We use the smallness of the solution to bound $K_{13}$,
$$
|K_{13}| \le C\, k_3 \, \|\tau\|_{L^\infty}\, \|\na \delta u\|_{L^2}^2 \le \frac{k_3}{16} \|\na \delta u\|_{L^2}^2,
$$
where we have used
$$
C\, \|\tau\|_{L^\infty} \le C\, \|\tau\|_{H^{s-1}} \le C\, \varepsilon \le \frac1{16}.
$$
$K_{14}$ is bounded similarly as $K_{12}$,
\beno
|K_{14}| &\le&  k_3\, \|\na \delta u\|_{L^2}\,
\|\delta \tau\|_{L^2}\, \|\na u^{(\nu)}\|_{L^\infty}\\
&\le&  \frac{k_3}{16} \|\na \delta u\|_{L^2}^2 + C\, k_3 \|u^{(\nu)}\|_{H^s}^2\,\|\delta \tau\|_{L^2}^2.
\eeno
\beno
|K_{15}| &\le&  k_3 \eta \, \|\Lambda^{2\beta} \delta \tau\|_{L^2} \,\|\na \delta u\|_{L^2} \\
&\le& \frac{\eta}{16} \|\Lambda^\beta\delta\tau\|^2_{H^\beta} + C\, k_3^2 \eta^{-1} \|\na \delta u\|_{L^2}^2\\
&\le& \frac{\eta}{16} \|\Lambda^\beta\delta\tau\|^2_{H^\beta} +  \frac{k_3}{16} \|\na \delta u\|_{L^2}^2.
\eeno
In addition, it is easy to obtain the following estimates
\beno
|K_{16}| \le \nu ^2 \|\Lambda^{2\alpha} u\|_{L^2}^2 + C\, \|\delta u\|_{L^2}^2,
\eeno
\beno
|K_{17}| &\le& C\,  \|u\|_{H^s} \, \|\delta u\|_{L^2}^2,
\eeno
\beno
|K_{18}| \le C\,  \|\tau\|_{H^s} \, (\|\delta u\|_{L^2}^2+\|\delta \tau\|_{L^2}^2),
\eeno
\beno
|K_{19}| \le C \|\tau\|^2_{H^s}\|\delta\tau\|^2_{L^2} + \frac{k_3}{16} \|\na \delta u\|_{L^2}^2,
\eeno
\beno
|K_{20}| \le C \|u^{(\nu)}\|_{H^s}\|\delta\tau\|^2_{L^2}.
\eeno
Inserting the bounds for $K_1$ through $K_{20}$ above in (\ref{jj10}), we find
\beno
&& \frac{d}{dt}(\|\delta u\|^2_{H^\beta}+\|\delta\tau\|^2_{H^\beta} + 2k_3(\delta u,\nabla\cdot\delta\tau)) \\
&& \qquad + 2 \nu \|\Lambda^{\alpha} \delta u\|^2_{H^\beta} +\frac{\eta}{4} \|\Lambda^\beta\delta\tau\|^2_{H^\beta} +
\frac{k_3}{4} \|\nabla\delta u\|_{L^2}^2 \\
&& \le  C(1 + \|u\|_{H^s}^2 + \|u^{(\nu)}\|_{H^s}^2 + \|\tau\|_{H^s}^2) (\|\delta u\|^2_{H^\beta}+\|\delta\tau\|^2_{H^\beta})\\
&&\quad  + C\, \nu^2 (\|u\|_{H^s}^2 + \|u^{(\nu)}\|_{H^s}^2).
\eeno
Choosing $k_3\le \frac12$, applying Gronwall's inequality and using the fact that $\|u^{(\nu)}\|_{H^s}$ is bounded uniformly in $\nu$ (see (\ref{unibound})), we obtain (\ref{invi1}). This completes the proof of Theorem \ref{main2}.
\end{proof}

\vskip .2in 
\textbf{Acknowledgments}.
The work of PC was partially supported by NSF grant DMS-1713985 and by the Simons Center for Hidden symmetries and Fusion Energy. 	The work of JW was partially supported by NSF 
grant DMS-1624146 and the AT\&T Foundation at Oklahoma State University. The work of JZ was partially supported by the National Natural Science Foundation of China (No.11901165). The work of YZ was partially supported by the National Natural Science Foundation of China (No. 11801175).

\end{document}